\documentclass[reqno]{amsart}
\usepackage{amsmath,amsthm,amscd,amssymb,amsfonts, amsbsy}
\usepackage{latexsym, color, enumerate}
\usepackage[mathscr]{eucal}

\usepackage[hmargin=2.54 cm,vmargin=2.54 cm]{geometry}

\theoremstyle{plain}
\newtheorem{theorem}[equation]{Theorem}
\newtheorem{lemma}[equation]{Lemma}

\theoremstyle{definition}

\theoremstyle{remark}
\newtheorem{remark}[equation]{Remark}

\numberwithin{equation}{section}

\newcommand{\bR}{\mathbb{R}}

\providecommand{\norm}[1]{\lVert#1\rVert}

\begin{document}

 \title[Liouville theorem for Navier-Stokes equations in weighted mixed-norm spaces]{Liouville type theorems for 3D stationary Navier-Stokes equations in weighted mixed-norm Lebesgue spaces}

\author[T. Phan]{Tuoc Phan}
\address[T. Phan]{Department of Mathematics, University of Tennessee, 227 Ayres Hall,
1403 Circle Drive, Knoxville, TN 37996-1320 }
\email{phan@math.utk.edu}

\thanks{T. Phan's research is partially supported by the Simons Foundation, grant \# 354889.}

\begin{abstract} This work studies the system of $3D$ stationary Navier-Stokes equations. Several Liouville type theorems are established for solutions in mixed-norm Lebesgue spaces and weighted mixed-norm Lebesgue spaces. In particular, we show that,  under some sufficient conditions in mixed-norm Lebesgue spaces, solutions of the stationary Navier-Stokes equations are identically zero. This result  covers the important case that solutions may decay to zero with different rates in different spatial directions, and some these rates could be significantly slow.  In the un-mixed norm case,  the result recovers available results. With some additional geometric assumptions on the supports of solutions,  this work also provides several other important Liouville type theorems for solutions in weighted mixed-norm Lebesgue spaces.  To prove the results, we establish some new results on mixed-norm and weighted mixed-norm estimates for Navier-Stokes equations.  All of these results are new and could be useful in other studies.
\end{abstract}

\maketitle

\today
\section{Introduction and main results}

This paper investigates the system of stationary Navier-Stokes equations of incompressible fluid in $3$-dimensional space
\begin{equation} \label{NS.eqn}
\left\{
\begin{array}{cccl}
-\Delta u + (u \cdot \nabla) u + \nabla p & = & 0 & \quad \text{in} \quad \bR^3, \\
\text{div}(u) & =& 0 & \quad \text{in} \quad \bR^3,
\end{array} \right.
\end{equation}
where $u =  (u_1, u_2, u_3): \bR^3 \rightarrow \bR^3$ is the unknown velocity of the considered fluid, and $p : \bR^3 \rightarrow  \bR$ is the unknown fluid pressure. Our main interests are to establish Liouville type theorems for solutions of the equations \eqref{NS.eqn} that may decay to zero in different rates as $|x| \rightarrow \infty$ in different directions.

To put our study in perspectives, let us recall some results regarding Liouville type theorems for solutions of the Navier-Stokes equations \eqref{NS.eqn}. A function 
$$u \in \dot{H}_\sigma^1(\bR^3) =\Big\{ v \in [\dot{H}^1(\bR^3)]^3: \text{div}[v] =0 \Big\} $$ 
is said to be  a $\dot{H}_\sigma^1(\bR^3)$-weak solution of \eqref{NS.eqn} if
\begin{equation} \label{test.eqn}
\sum_{k=1}^3\int_{\bR^3}\Big[ \nabla u_k \cdot \nabla \varphi_k  + (u\cdot \nabla) u_k \phi_k\Big] dx - \int_{\bR^3} p \text{div} \phi dx =0 \quad \text{for all} \  \phi \in [C_0^\infty(\bR^3)]^3,
\end{equation}
where $H^1(\bR^3)$ denotes the usual Sobolev space and the pressure $p$ is defined by
\begin{equation} \label{pressure-formula}
p =\sum_{i, j =1}^3 \mathscr{R}_i \mathscr{R}_j (u_i u_j),
\end{equation}
in which $\mathscr{R}_i $ denotes the $i$-th Riesz transform. As $u \in \dot{H}^1_\sigma(\bR^3)$, by the Sobolev imbedding, we see that $u \in L_6(\bR^3)$. From this and the regularity theory (see \cite[Chapter 3]{Seregin-book}  for instance), we can infer that  $u \in C^\infty(\bR^3) \cap L_\infty(\bR^3)$ and
\[
|u(x)| \rightarrow 0 \quad \text{uniformly as} \ |x| \rightarrow \infty. 
\]
Obviously, zero is a $\dot{H}_\sigma^1(\bR^3) $-weak solution of \eqref{NS.eqn}. Due to various interests, it is an important  question to ask if zero is the only $\dot{H}_\sigma^1(\bR^3) $-weak solution of \eqref{NS.eqn}, see \cite[Remark X.9.4, p. 729]{Galdi} and \cite[Conjecture 2.5, p. 23]{Tsai}. 

Great efforts have been invested (for instance, see \cite{Chae-2018, Chae-Wolf, CJL-R, KTW, Seregin-1916, Seregin-2018, Seregin-Wang}) to study the above uniqueness question. However,  the question outstandingly remains to be open.  Various partial progresses are made to understand this uniqueness problem. One of the first results is due to \cite[Theorem X.9.5, p. 729]{Galdi}. Esentially, it is proved in \cite[Theorem X.9.5, p. 729]{Galdi} that if $u \in L_{\frac{9}{2}}(\bR^3)$, then $u \equiv 0$. Note that if $u \in \dot{H}^1_\sigma(\bR^3)$, then by the Sobolev's imbedding $u \in L_6(\bR^3)$. Therefore, there is a great discrepancy of the result in the mentioned uniqueness problem. Many other interesting results that extend this result can be found in \cite{Chae-Wolf, CJL-R, KTW, Seregin-1916, Seregin-2018} and the references therein. Notably, it is proved in \cite{Seregin-1916} that the same conclusion holds true for solutions which are assumed to be in $ L_6(\bR^3) \cap \textup{BMO}^{-1}(\bR^3)$. See also \cite{Chae-2018, Seregin-1916, Seregin-Wang} for some recent extensions of this result.

In this paper, we study the uniqueness problem and extend the mentioned known results to a completely new direction. In particular,  we investigate solutions of the equations \eqref{NS.eqn} that may decay to zero in different rates as $|x| \rightarrow \infty$ in different directions. We follow the spirit of the work  \cite{Krylov} to use mixed-norm Lebesgue spaces to measure those kinds of such functions, see also \cite{Dong-Kim, D-P}.  For two given numbers $q, r \in (1, \infty)$, the mixed-norm Lebesgue space $L_{q,r}(\bR^3)$  is the space that is equipped with the following norm
\[
\norm{f}_{L_{q,r}(\bR^3)} = \left [\int_{\bR} \left (\int_{\bR^2} |f(x_1, x_2, x_3)|^q dx_1 dx_2 \right)^{r/q} dx_3 \right]^{1/r},
\]
where $f : \mathbb{R}^3 \rightarrow \mathbb{R}$ is  a measurable function.   When $q=r$, we just simply write $L_{q}(\bR^3) = L_{q,q}(\bR^3)$.  In this paper, a function $f$ is said to be in $L_{q,\text{loc}}(\bR^3)$ if $f \in L_q(U)$ for every compact set $U \subset \bR^3$.  In a similar way, we also write $f \in H^1_{\text{loc}}(\bR^3)$ if $f \in H^1(U)$ for every compact set $U \subset \bR^3$.  Also, throughout the paper, a given function $u \in [H^1_{\text{loc}}(\bR^3)]^3$ is said to be a weak solution of \eqref{NS.eqn} if $\text{div}[u] =0$ in the distributional sense,  \eqref{pressure-formula} is well-defined in some suitable sense so that $p \in L_{1, \text{loc}}(\bR^3)$  and \eqref{test.eqn} holds.

The first  result of this paper is the following Liouville theorem  for solutions of the Navier-Stoke equations \eqref{NS.eqn} in mixed norm Lebesgue spaces.
\begin{theorem} \label{mixed-norm} Let $q, r \in [3, \infty)$ be two numbers satisfying 
\begin{equation} \label{q-r-relation}
\frac{2}{q} + \frac{1}{r} \geq \frac{2}{3}.
\end{equation}
Also, let  $u \in [H^1_{\textup{loc}}(\bR^3)]^3$ be a weak solution of \eqref{NS.eqn} and assume that $u \in [L_{q,r}(\bR^3)]^3$. Then $u \equiv 0$ in $\bR^3$. 
\end{theorem}
\begin{remark}  It is very interesting to observe from \eqref{q-r-relation} that either $q$ or $r$ can be taken to be sufficiently large. Consequently, the solution $u$ could decay to zero sufficiently slow as $|x| \rightarrow \infty$ and it may not be in neither  $L_{\frac{9}{2}}(\bR^3)$ nor $L_6(\bR^3)$.  Therefore, Theorem \ref{mixed-norm} covers the cases that are not covered in the known work such as \cite{Chae-2018, Chae-Wolf, CJL-R, KTW, Seregin-1916, Seregin-2018, Seregin-Wang}. It maybe also of great interest to find in Remark \ref{weighted-thrm-rm} below for a variant of Theorem \ref{mixed-norm} in weighted  Lebesgue spaces.  Note that Theorem \ref{mixed-norm} holds for $q = r = \frac{9}{2}$ and therefore it recovers the result established in  \cite[Theorem X.9.5, p. 729]{Galdi}. In the case that $p = r$ and $q \in [3, \frac{9}{2}]$, Theorem \ref{mixed-norm} also recovers the recent results obtained in \cite{CJL-R}.
\end{remark}
In some special cases where we have additional geometric assumptions on the supports of solutions, the ranges of the numbers $q$ and $r$ in Theorem \ref{mixed-norm} are improved significantly.  Our next two results are Liouville type theorems for  \eqref{NS.eqn} in this spirit. To introduce the results, we need to introduce the weighted mixed-norm Lebesgue spaces. For given numbers $q, r \in (1, \infty)$, and for two given weight functions
\[
\omega_1 : \bR^2 \rightarrow \bR \quad \text{and} \quad \omega_2: \mathbb{R} \rightarrow \mathbb{R},
\] 
the measurable function $f : \mathbb{R}^3 \rightarrow \mathbb{R}$ is said to be in the weighted mixed-norm Lebesgue space $L_{q,r}(\bR^3, \omega)$ if
\[
\norm{f}_{L_{q,r}(\bR^3, \omega)} = \left [\int_{\bR} \left (\int_{\bR^2} |f(x_1, x_2, x_3)|^q \omega_1(x_1, x_2) dx_1 dx_2 \right)^{r/q} \omega_2(x_3) dx_3 \right]^{1/r} <\infty,
\]
where $\omega(x) = \omega(x_1, x_2) \omega(x_3)$ for a.e. $x = (x_1, x_2, x_3) \in \bR^3$. Our second result is the following Liouville type theorem in weighted mixed-norm Lebesgue spaces for solutions whose supports are in strips in $\bR^3$. 
\begin{theorem} \label{strip-thrm} Let $q \in [3, 6]$, $r \in [3, \infty)$, $\alpha = \frac{6-q}{3}$ be fixed numbers, and  let 
$$ \omega_1(x') = (1 + |x'|)^{-\alpha}, \quad x' = (x_1, x_2) \in \bR^2$$
and $\omega_2 \equiv 1$ in $\bR$.  Also, let $u \in [H^1_{\textup{loc}}(\bR^3)]^3$ be a weak solution of \eqref{NS.eqn}.  Assume that $u\in [L_{q, r}(\bR^3, \omega)]^3$  and there exists $R_0 >0$ such that
\[
\textup{support}(u) \subset  \mathbb{R}^2 \times [-R_0, R_0].
\]
Then, $u \equiv 0$ in $\bR^3$.
\end{theorem}
\begin{remark} Observe that Theorem \ref{strip-thrm} holds with $p = r =6$. This result demonstrates a possibility that  Liouville theorem for the Navier-Stokes equations \eqref{NS.eqn} holds for $L_6(\bR^3)$-solutions. See also a similar result in \cite{Seregin-1916} in which the solutions are assumed to be in $ L_6(\bR^3) \cap \textup{BMO}^{-1}(\bR^3)$.
\end{remark}
The last result in this paper is a Liouville type theorem  in weighted mixed norm spaces for solutions whose supports are in cylinders in $\bR^3$. For this purpose, for each $R >0$, we denote the cylinder along the $x_3$-axis  in $\bR^3$ of radius $R$ by
\[
C_R = B_R' \times \mathbb{R},  
\]
where $B_R'$ is the ball in $\bR^2$ centered at the origin with radius $R$. Our result is the following Liouville type theorem for solutions in weighted mixed-norm Lebesgue spaces in $\bR^3$. 
\begin{theorem} \label{cyl-thrm}   Let $q, r \in [3, \infty)$, $\alpha \in [0,1)$ be fixed numbers,  and let $\omega_1 \equiv 1$ in $\bR^2$ and $\omega_2(x_3) =(1+ |x_3|)^{-\alpha}$ for $x_3 \in \bR$.  Also, let $u \in [H^1_{\textup{loc}}(\bR^3)]^3$ be a weak solution of \eqref{NS.eqn}. Assume that  $u\in [L_{q,r}(\bR^3, \omega)]^3$ and there exists $R_0 >0$ such that
\[
\textup{support}(u) \subset  C_{R_0}.
\]
Then, $u \equiv 0$ in $\bR^3$.
\end{theorem}
\begin{remark} Theorem \ref{cyl-thrm} is interesting because it allows $u$ to decay to zero in $x_3$ at very slow rate. More precisely, let us define
\[
\psi(x_3) = \left(\int_{B'_{R_0}} |u(x', x_3)|^{q} dx' \right)^{\frac{1}{q}}.
\]
Assume that there exist $N_0 >0$ and $\beta >0$ such that
\[
|\psi(x_3)| \leq \frac{N_0}{(1 + |x_3|)^{\beta}}, \quad \forall \ x_3 \in \bR.
\]
Then, with the choice of $\alpha \in (0,1)$ and sufficiently close to $1$ so that $\beta r+\alpha >1$, we see that $\psi \in L_r(\bR, \omega_2)$ and therefore $u \in L_{q,r}(\bR^3, \omega)$. As $\beta$ can be sufficiently small, $u \not\in L_{\frac{9}{2}}(\bR^3)$ and also $u \not\in L_6(\bR^3)$ and therefore the results in the available work such as \cite{Chae-Wolf, CJL-R, KTW, Seregin-1916, Seregin-2018} are not applicable.
\end{remark}
The remaining part of the paper is to prove Theorems \ref{mixed-norm}, \ref{strip-thrm},  \ref{cyl-thrm}. Our approach is based on the combination of the approach used in \cite[Theorem X.9.5, p. 729]{Galdi} together with some new results on mixed-norm estimates and weighted mixed-norm estimates (see Lemma \ref{pressure-lemma} below).  Though, the proofs of these theorems are semilar, but details calculations are needed to be adjusted differently. We therefore provide the details of the proof of Theorem \ref{mixed-norm}. Meanwhile, we also give essential steps in the proofs of Theorem \ref{strip-thrm} and \ref{cyl-thrm}.  In Section \ref{weights}, we recall the definitions of Muckenhoupt classes of weights, and we introduce a lemma on weighted mixed norm estimates for the pressure $p$ in \eqref{NS.eqn}. The proof of Theorem \ref{mixed-norm} is given in Section \ref{mixed-norm-sec}. Section \ref{strip-ses} is about the proof of Theorem \ref{strip-thrm}, meanwhile the last section, Section \ref{cyl-section},  is to prove Theorem \ref{cyl-thrm}.

\section{Weights and weighted mixed-norm estimates} \label{weights}
This section recalls some definitions and deduces some weighted mixed-norm estimates for the pressure.  For each $q \in [1, \infty)$, a non-negative measurable function $\omega: \mathbb{R}^n \rightarrow \mathbb{R}$ is said to be in the Muckenhoupt  $A_q(\bR^n)$-class of weights if
\[
[\omega]_{A_q} := \displaystyle{\sup_{R>0 , x_0 \in \bR^n} \left(\frac{1}{|B_R(x_0)|} \int_{B_R(x_0)} \omega(x) dx \right) \left(\frac{1}{B_R(x_0)} \int_{B_R(x_0)} \omega(x)^{-\frac{1}{q-1}} dx  \right)^{q-1}} < \infty  
\]
for  $q \in (1, \infty)$, and
\[
[\omega]_{A_1} := \displaystyle{\sup_{R>0 , x_0 \in \bR^n} \left(\frac{1}{|B_R(x_0)|} \int_{B_R(x_0)} \omega(x) dx \right)  \|\frac{1}{\omega}\|_{L_\infty(B_R(x_0)}} < \infty,
\]
where $B_R(x_0)$ denotes the ball in $\bR^3$ of radius $R$ and centered at $x_0 \in \bR^3$. We  also recall that for two given numbers $q, r \in (1, \infty)$, and for two weight functions
\[
\omega_1 : \bR^2 \rightarrow \bR \quad \text{and} \quad \omega_2: \mathbb{R} \rightarrow \mathbb{R},
\] 
the measurable function $f : \mathbb{R}^3 \rightarrow \mathbb{R}$ is said to be in the weighted mixed-norm Lebesgue space $L_{q,r}(\bR^3, \omega)$ if
\[
\norm{f}_{L_{q,r}(\bR^3, \omega)} = \left [\int_{\bR} \left (\int_{\bR^2} |f(x_1, x_2, x_3)|^q \omega_1(x_1, x_2) dx_1 dx_2 \right)^{r/q} \omega_2(x_3) dx_3 \right]^{1/r} <\infty,
\]
where $\omega(x) = \omega_1(x_1, x_2) \omega_2(x_3)$ for all $x = (x_1, x_2, x_3) \in \bR^3$.  In this paper, at various contexts, with a given weight function $\omega: \bR^3 \rightarrow \bR$  we also write $L_{q}(\bR^3,\omega)$ for the usual weighted Lebesgue space whose norm is defined by
\[
\norm{f}_{L_q(\bR^3, \omega)} = \left( \int_{\bR^3} |f(x)|^q \omega(x) dx \right)^{\frac{1}{q}}.
\]

We introduce the following lemma on weighted mixed norm estimates for the pressure $p$ of the equations \eqref{NS.eqn}. This lemma is an important ingredient in this paper.
\begin{lemma} \label{pressure-lemma} Let $q, r \in (2, \infty)$ and $M_0 \geq 1$. Assume that $u \in [L_{q, r}(\bR^3, \omega)]^3$ with $\omega(x) =\omega_1(x') \omega_2(x_3)$ for all a.e. $x = (x', x_3) \in \bR^2 \times \bR$, and for $\omega_1 \in A_{\frac{q}{2}}(\bR^2)$, $\omega_2 \in A_{\frac{r}{2}}(\bR)$ with 
\[
[\omega_1]_{A_{\frac{q}{2}}(\bR^2)} \leq M_0, \quad [\omega_2]_{A_{\frac{r}{2}}(\bR)} \leq M_0.
\]
Then, there exists $N = N(q,r, M_0) >0$ such that
\[
\norm{p}_{L_{\frac{q}{2}, \frac{r}{2}}(\bR^3, \omega)} \leq N \norm{u}_{L_{q, r}(\bR^3, \omega)}^2,
\]
where $p$ is defined as in \eqref{pressure-formula}.
\end{lemma}
\begin{proof} We use the idea developed in \cite{Dong-Kim} which makes use of the extrapolation theory due to Rubio de Francia (see \cite{David} and also \cite{Dong-Kim}). Recall that
\[
p =\sum_{i, j =1}^3 \mathscr{R}_i \mathscr{R}_j (u_i u_j),
\]
where $\mathscr{R}_i $ denotes the $i$-th Riesz transform. Recall also for each $l \in (1, \infty)$, $M_0 \geq 1$, and each $\mu \in A_l(\bR^3)$ with $[\mu]_{A_l(\bR^3)} \leq M_0$ the map $\mathscr{R}_i \mathscr{R}_j$ is bounded from $L_l(\bR^3, \mu) \rightarrow L_l(\bR^3, \mu)$ and (see \cite{RF}, for example) 
\begin{equation} \label{Riesz-1}
\norm{\mathscr{R}_i \mathscr{R}_j (f)}_{L_l(\bR^3, \mu)} \leq N(l, M_0) \norm{f}_{L_l(\bR^3, \mu)}.
\end{equation}
It is sufficient to show that the estimate \eqref{Riesz-1} can be extended to the weighted mixed norm. In particular, we claim that for every $q_1, q_2 \in (1, \infty)$
\begin{equation} \label{Riesz-2}
\norm{\mathscr{R}_i \mathscr{R}_j (f)}_{L_{q_1, q_2} (\bR^3, \mu)} \leq N(q_1, q_2, M_0) \norm{f}_{L_{q_1,q_2}(\bR^3, \mu)}.
\end{equation}
for $\mu(x) = \mu_1(x') \mu_2(x)$ with $x =(x',x_3) \in \bR^2\times \bR$ and
\begin{equation} \label{mu-1-2}
[\mu_1]_{A_{q_1}(\bR^2)} \leq M_0, \quad [\mu_2]_{A_{q_2}(\bR)} \leq M_0.
\end{equation}
To this end, for a fixed $\mu_1 \in A_{q_1}(\bR^2)$ as in \eqref{mu-1-2}, let us define
\[
\begin{split}
&  \phi(x_3) = \left(\int_{\bR^2} |\mathscr{R}_i \mathscr{R}_j (f)(x', x_3)|^{q_1} \mu_1(x') dx' \right)^{\frac{1}{q_1}} \quad \text{and} \\
& \psi(x_3) = \left(\int_{\bR^2} |f(x', x_3)|^{q_1} \mu_1(x') dx'\right)^{\frac{1}{q_1}}, \quad \text{for a.e.} \ x_3 \in \bR.
\end{split}
\]
Let $\tilde{\omega} \in A_{q_1}(\bR)$ be any weight and we denote
\[
\tilde{\mu}(x) = \mu_1(x') \tilde{\omega}(x_3), \quad x =(x',x_3) \in \bR^2\times \bR.
\]
Observe that $\tilde{\mu} \in A_{q_1}(\bR^3)$. Therefore, it follows from \eqref{Riesz-1} that
\[
\norm{\phi}_{L_{q_1}(\bR, \tilde{\omega})} = \norm{\mathscr{R}_i \mathscr{R}_j (f)}_{L_{q_1}(\bR^3, \tilde{\mu})} \leq N \norm{f}_{L_{q_1}(\bR^3, \tilde{\mu})} = N \norm{\psi}_{L_{q_1}(\bR, \tilde{\omega})}.
\]
Then, by the extrapolation theorem (see \cite[Theorem 2.5]{Dong-Kim}, for intance), we infer that
\[
\norm{\phi}_{L_{q_2}(\bR, \tilde{\omega})} \leq N \norm{\psi}_{L_{q_2}(\bR, \tilde{\omega})}
\]
for every $q_2 \in (1, \infty)$ and for $\tilde{\omega} \in A_{q_2}(\bR)$. This implies \eqref{Riesz-2} and  the proof of the lemma is completed. 
\end{proof}
\begin{remark} Another possible way of proving Lemma \ref{pressure-lemma} is to consider $u$ as is a weak solution of the Stokes system
\[
\left\{
\begin{array}{cccl}
-\Delta u + \nabla p & = & -\text{div} (u\otimes u) & \quad \text{in} \quad \bR^3 \\
\text{div}(u) & = & 0 & \quad \text{in} \quad \bR^3.
\end{array} \right.
\]
Then, it is possible to modify the mixed-norm regularity results obtained in \cite{Dong-Kim, D-K-18, D-P, Krylov} for the Stokes system of equations to deduce that
\[
\norm{\nabla \otimes u}_{L_{\frac{p}{2}, \frac{r}{2}}(\bR^3, \omega)} +  \norm{p}_{L_{\frac{p}{2}, \frac{r}{2}}(\bR^3, \omega)} \leq N(p, r, M_0)\norm{|u|^2}_{_{L_{\frac{p}{2}, \frac{r}{2}}(\bR^3, \omega)}} = N \norm{u}_{_{L_{p, r}(\bR^3, \omega)}}^2.
\]
We plan to come back to this issue in our future work.
\end{remark}
To conclude the section, we include the following simple but important lemma which enables us to use the weighted mixed-norm estimates in Lemma \ref{pressure-lemma} with the weights defined as in Theorem \ref{strip-thrm} and Theorem \ref{cyl-thrm}.
\begin{lemma} \label{remark-weight} Let $q \in (1, \infty)$,   $\alpha \in [0,n)$, and $\omega(x) = (1+|x|)^{-\alpha}$ for all $x \in \bR^n$. Then, $\omega \in A_q(\bR^n)$.
\end{lemma}
\begin{proof} The proof is elementary and we sketch a few steps for completeness. First of all,  we claim that
\begin{equation} \label{origin-balls}
M_0 := \sup_{R >0} \left(\frac{1}{|B_R|}\int_{B_R} \omega(x) dx \right) \left(\frac{1}{|B_R|}\int_{B_R}\omega(x)^{-\frac{1}{q-1}}\right)^{q-1} <\infty,
\end{equation}
where $B_R$ denotes the ball of radius $R$ in $\bR^n$ centered at the origin.  As $\omega(x)$ and $\omega(x)^{-\frac{1}{q-1}}$ are bounded in compact sets, to prove \eqref{pressure-lemma}, we only need to consider the case that $R$ is large, saying $R>1000$. In this case, observe that
\[
\begin{split}
\frac{1}{|B_R|}\int_{B_R} \omega(x) dx  & = R^{-n} \int_{0}^R (1+r)^{-\alpha} r^{n-1} dr \\
& = R^{-n} \left[  \int_{1000}^R (1+r)^{-\alpha} r^{n-1} dr +  \int_{0}^{1000} (1+r)^{-\alpha} r^{n-1} dr \right] \\
& \leq N (n, \alpha) R^{-n }\left[\int_{1000}^R (1+r)^{n-\alpha-1}  dr + 1 \right] \\
&\leq N(n, \alpha) (1+R)^{-\alpha}.
\end{split}
\]
Similarly,
\[
\frac{1}{|B_R|} \int_{B_R} \omega(x)^{-\frac{1}{q-1}} dx = R^{-n} \int_0^R (1+r)^{\frac{\alpha}{q-1}} r^{n-1}dr \leq N(n, \alpha) (1+R)^{\frac{\alpha}{q-1}}.
\]
Consequently, 
\[
\begin{split}
\sup_{R>1000} \left(\frac{1}{|B_R|}\int_{B_R} \omega(x) dx \right) \left(\frac{1}{|B_R|}\int_{B_R}\omega(x)^{-\frac{1}{q-1}}\right)^{q-1}  \leq N(n, \alpha) < \infty
\end{split}
\]
and \eqref{origin-balls} follows.

Now with \eqref{origin-balls} in hand, we prove 
\[
\sup_{R >0, x_0 \in \bR^n} \left(\frac{1}{|B_R(x_0)|}\int_{B_R(x_0)} \omega(x) dx \right) \left(\frac{1}{|B_R(x_0)|}\int_{B_R(x_0)}\omega(x)^{-\frac{1}{q-1}}\right)^{q-1} <\infty
\]
by considering two cases.\\
\noindent {\bf Case 1}. We assume that $|x_0| \leq 3R$. In this case, observe that $B_{R}(x_0) \subset B_{4R}$. Then, by \eqref{origin-balls}, we see that 
\[
\begin{split}
& \left(\frac{1}{|B_R(x_0)|}\int_{B_R(x_0)} \omega(x) dx \right) \left(\frac{1}{|B_R(x_0)|}\int_{B_R(x_0)}\omega(x)^{-\frac{1}{q-1}}\right)^{q-1}\\
& \leq N(n, q) \left(\frac{1}{|B_{4R}|}\int_{B_{4R}} \omega(x) dx \right) \left(\frac{1}{|B_{4R}|}\int_{B_{4R}}\omega(x)^{-\frac{1}{q-1}}\right)^{q-1} \\
& \leq N(n, q) M_0.
\end{split}
\]
\noindent {\bf Case 3}. We assume that $|x_0| > 3R$. In this case, we see that 
\[
\begin{split} 
1+ |x| & \geq 1+ |x_0| - |x-x_0|  \geq 1+ |x_0| - R \\
& \geq \frac{1 + |x_0|+R}{2}, \quad \forall \ x \in B_{R}(x_0). 
\end{split}
\]
Therefore,
\[
\frac{1}{|B_{R}(x_0)|}\int_{B_{R}(x_0)} \omega(x) dx \leq N(n, \alpha)  [1+ |x_0|  + R]^{-\alpha}.
\]
On the other hand, we also have
\[
1+ |x| \leq 1 + |x_0| + |x-x_0| \leq 1+ |x_0| +R, \quad \forall \ x \in B_R(x_0).
\]
Then,
\[
\left(\frac{1}{|B_R(x_0)|}\int_{B_R(x_0)}\omega(x)^{-\frac{1}{q-1}}\right)^{q-1}\leq N(n)[1+|x_0| +R]^{\alpha}.
\]
Hence,
\[
\left(\frac{1}{|B_R(x_0)|}\int_{B_R(x_0)} \omega(x) dx \right) \left(\frac{1}{|B_R(x_0)|}\int_{B_R(x_0)}\omega(x)^{-\frac{1}{q-1}}\right)^{q-1} \leq N(n, \alpha)
\]
and this completes the proof.
\end{proof}
\noindent
\section{Liouville theorems for solutions in mixed norm Lebesgue spaces in $\bR^3$}  \label{mixed-norm-sec}
This section proves Theorem \ref{mixed-norm}. We follow the approach used in the proof of \cite[Theorem X.5.1, p. 729]{Galdi} and combine it with our new ingredient on weighted mixed-norm estimates in Lemma \ref{pressure-lemma}. 
\begin{proof}[Proof of Theorem \ref{mixed-norm}] For each $R>0$, we denote the cube in $\bR^3$ centered at the origin with radius $R$ by
\[
Q_R = [-R, R]^3.
\]
Let $\phi \in C^\infty_0(\bR)$ be a standard cut-off function with $0 \leq \phi \leq 1$ and
\[
\phi =1 \quad \text{on} \quad [-1/2, 1/2] \quad \text{and} \quad \phi =0 \quad \text{on} \quad \bR \setminus [-1, 1].
\] 
For each $R>0$,  let 
$$\phi_R(x) = \phi(\frac{x_1}{R}) \phi(\frac{x_2}{R})\phi(\frac{x_3}{R}), \quad x=(x_1, x_2, x_3) \in \bR^3.$$
Then, we see that
\[
\phi_R(x) = 1 \quad \text{on} \quad Q_{R/2}, \quad \phi =0 \quad \text{on} \quad \bR^3\setminus Q_R.
\]
Moreover, there is a universal constant $N_0$ independent on $R$ such that
\begin{equation} \label{grad-test-fn}
|\nabla \phi_R| \leq \frac{N_0}{R} \quad \text{and} \quad |\nabla^2 \phi_R|  \leq \frac{N_0}{R^2} \quad \text{on} \quad \bR^3.
\end{equation}
Note that it follows from Lemma \ref{pressure-lemma} that $p \in L_{\frac{p}{2}, \frac{r}{2}}(\bR^3)$. Observe also that as $u \in [H^1_{\text{loc}}(\bR^3)]^3$, we see that $u \in [L_{6,\text{loc}}(\bR^3)]^3$ by the Sobolev's imbedding (in fact, by the regularity theory (see \cite[chapter 3]{Seregin-book}), $u \in C^\infty(\bR^3)$). Therefore, we can use  $u\phi_R$ as a test function in \eqref{test.eqn} for the system \eqref{NS.eqn} and obtain
\begin{equation} \label{1-indentity}
\sum_{k=1}^3 \int_{\bR^3}  \nabla u_k \cdot \nabla (\phi_R u_k) dx + \int_{\bR^3} (u \cdot \nabla u) \cdot (\phi_R u) dx - \int_{\bR^3} p \text{div}(\phi_R u) dx =0.
\end{equation}
We now manipulate each term in \eqref{1-indentity}. With the integration by parts,  the first term can be rewritten as
\begin{align} \nonumber
 \sum_{i=1}^3 \int_{Q_R} (\nabla u_i)  \cdot \nabla (\phi_R u_i) dx  & = \sum_{i=1}^3\int_{Q_R}\Big[u_i (\nabla u_i) \cdot (\nabla \phi_R)  + |\nabla u_i|^2 \phi_R  \Big] dx \\ \nonumber
& = \int_{Q_R} \nabla \Big(\frac{|u|^2}{2} \Big) \cdot \nabla \phi_R dx + \int_{Q_R} |\nabla \otimes u|^2 \phi_R dx \\ \label{term-1}
& = - \frac{1}{2}\int_{Q_R \setminus Q_{R/2}} |u|^2 \Delta \phi_R dx + \int_{Q_R} |\nabla \otimes u|^2 \phi_R dx.
\end{align}
Regarding the second term in \eqref{1-indentity}, we first observe that it can be rewritten as
\[
\int_{\bR^3} (u\cdot \nabla u) \cdot (\phi_R u) dx  = \int_{Q_R} \phi_R u \cdot \nabla \Big( \frac{|u|^2}{2}\Big) dx.
\]
Then, with the fact that $\textup{div}(u) =0$, we can use the integration by parts and then obtain
\begin{equation} \label{term-2}
\int_{\bR^3} (u\cdot \nabla u) \cdot (\phi_R u) dx = - \frac{1}{2} \int_{Q_R \setminus Q_{R/2}} |u|^2 \Big [ u \cdot \nabla \phi_R \Big]dx.
\end{equation}
For the last term in \eqref{1-indentity}, we again use the fact that $\textup{div}(u) =0$ to manipulate it as follows
\begin{equation}  \label{term-3}
\int_{\bR^3}  p \text{div} (\phi_R u) dx  = \int_{Q_R \setminus Q_{R/2}} p \Big[ u \cdot \nabla \phi_R \Big]dx.
\end{equation}
From the manipulations in \eqref{term-1}, \eqref{term-2} and \eqref{term-3}, we can rewrite the identity \eqref{1-indentity} as 
\begin{equation*}
\int_{Q_R} |\nabla \otimes u|^2 \phi_R(x) dx = \frac{1}{2} \int_{Q_R\setminus Q_{R/2}} |u|^2 \Delta \phi_R dx + \int_{Q_R\setminus Q_{R/2}} \Big( u \cdot \nabla \phi_R\Big) \Big[\frac{ |u|^2}{2} +p \Big] dx.
\end{equation*}
From this, we infer the following important Cacciopolli type estimate
\begin{align}  \nonumber
\int_{Q_{R/2}} |\nabla \otimes u|^2 dx & \leq \frac{1}{2}\int_{Q_R\setminus Q_{R/2}} |u|^2 |\Delta \phi_R| dx + \frac{1}{2} \int_{Q_R\setminus Q_{R/2}} |u|^3  |\nabla \phi_R| dx  \\ \label{2-identity}
& \quad \quad + \int_{Q_R\setminus Q_{R/2}} |p| |u| |\nabla \phi_R| dx.
\end{align}
We now respectively denote $I_1(R),  I_2(R)$, and $I_3(R)$ the first, the second, and the last term on the right hand side of \eqref{2-identity}. The rest of the proof is to control these terms. To control term $I_1$, we choose $q_1, r_1 \in (1, \infty)$ that satisfy
\begin{equation} \label{q-1-r-1.def}
 \frac{2}{q} + \frac{1}{q_1} =1 \quad \text{and} \quad \frac{2}{r} + \frac{1}{r_1} =1.
 \end{equation}
Then, we use H\"{o}lder's inequality with the exponents $\frac{q}{2}$ and $q_1$ for the integration in $x' = (x_1, x_2)$-variable, and then use the H\"{o}lder's inequality with the exponents $\frac{r}{2}$ and $r_1$ for the intergation in $x_3$-variable. From those calculations, we  infer that
\begin{align*} \nonumber
I_1(R) & \leq \frac{1}{2}\norm{\Delta \phi_R}_{L_{q_1, r_1} (Q_R)} \norm{u}_{L_{q,r}(Q_R\setminus Q_{R/2})}^2 \leq N R^{\frac{2}{q_1} + \frac{1}{r_1} -2} \norm{u}_{L_{q,r}(Q_R\setminus Q_{R/2})}^2 \\ 
& = N R^{1- 2\big(\frac{2}{q} + \frac{1}{r} \big)} \norm{u}_{L_{q,r}(Q_R\setminus Q_{R/2})}^2.
\end{align*}
where in the second step in the above calculation, we used the first estimate in \eqref{grad-test-fn}. By our assumption, we see that  $1- 2\big(\frac{2}{q} + \frac{1}{r} \big) \leq 0$.  From this, and since  $u \in L_{q,r}(\bR^3)$, we can infer from the last estimate that
\begin{equation} \label{I-1-term}
\lim_{R\rightarrow \infty} I_1(R) =0,
\end{equation}
Next, we control $I_2(R)$. As $q, r \in [3, \infty)$, we can choose the numbers $q_2, r_2 \in (1,\infty]$  such that
\begin{equation} \label{q-2-r-2}
\frac{3}{q} + \frac{1}{q_2} =1 \quad \text{and}  \quad \frac{3}{r} + \frac{1}{r_2} =1.
\end{equation}
If $q_2 < \infty$ and $r_2 <\infty$, as in the previous step of controlling $I_1(R)$, we apply H\"{o}lder's inequality for the integration in $x'$-variable and then  H\"{o}lder's inequality  for the integration in $x_3$-variable to obtain
\[
\begin{split}
I_{2}(R) & \leq N \norm{\nabla \phi_R}_{L_{q_2, r_2}(Q_R)} \norm{u}_{L_{q,r}(Q_R\setminus Q_{R/2})}^3 \leq N R^{\frac{2}{q_2} + \frac{1}{r_2} -1} \norm{u}_{L_{q,r}(Q_R\setminus Q_{R/2})}^3  \\
& = N R^{2 -3\big(\frac{2}{q} + \frac{1}{r}\big)} \norm{u}_{L_{q,r}(Q_R\setminus Q_{R/2})}^3 .
\end{split}
\]
Observe also that when $q_2 = \infty$ or $r_2 =\infty$, the above estimate also holds. Now, from our assumptions, we see that $2 -3\big(\frac{2}{q} + \frac{1}{r}\big) \leq 0$ and $u \in L_{q,r}(\bR^3)$. Therefore, we conclude that
\begin{equation}\label{I-2-term}
\lim_{R\rightarrow \infty} I_2(R) =0.
\end{equation}
Finally, we control $I_3(R)$. For  $q_2$ and $r_2$ defined in \eqref{q-2-r-2}, we observe that
\begin{equation*} 
\frac{1}{q_2} + \frac{2}{q} + \frac{1}{q} =1, \quad \frac{1}{r_2} + \frac{2}{r} + \frac{1}{r} =1.
\end{equation*}
Therefore, in case that $q_2 < \infty$ and $r_2 <\infty$ we can apply the H\"{o}lder's inequality for the integration in $x'$-variable using the three exponents $q_2, \frac{q}{2}$ and $q$. Then, we also apply  H\"{o}lder's inequality for the integration in $x_3$-variable using the three exponents $r_2, \frac{r}{2}$ and $r$.  In case that $q_2 =  \infty$ or $r_2 =\infty$ we can perform a similar estimate. As a result, we obtain
\[
\begin{split}
I_3(R) & \leq N(q,r) \norm{\nabla \phi_R}_{L_{q_2, r_2}(Q_R\setminus Q_{R/2})} \norm{p}_{L_{\frac{q}{2}, \frac{r}{2}}(Q_R\setminus Q_{R/2})} \norm{u}_{L_{p,r}(Q_{R} \setminus Q_{R/2})}\\
& \leq N(q,r)R^{2-3\big(\frac{2}{q} +\frac{1}{r}\big)}\norm{p}_{L_{\frac{q}{2}, \frac{r}{2}}(\bR^3)} \norm{u}_{L_{p,r}(Q_{R} \setminus Q_{R/2})}.
\end{split}
\]
Then, using Lemma \ref{pressure-lemma}, we deduce that
\[
I_3(R) \leq N(q,r)R^{2-3\big(\frac{2}{q} +\frac{1}{r}\big)}\norm{u}_{L_{q, r}(\bR^3)} \norm{u}_{L_{p,r}(Q_{R} \setminus Q_{R/2})}.
\]
Now, as $2-3\big(\frac{2}{q} +\frac{1}{r}\big) \leq 0$ and $u \in L_{q,r}(\bR^3)$, we conclude that
\begin{equation} \label{I-3-term}
\lim_{R\rightarrow \infty} I_{3}(R) =0.
\end{equation}
Now, by collecting the estimates \eqref{2-identity}, \eqref{I-1-term}, \eqref{I-2-term}, and \eqref{I-3-term}, we obtain
\[
\int_{\bR^3} |\nabla \otimes u|^2 dx = \lim_{R\rightarrow \infty} \int_{Q_{R/2}}|\nabla \otimes u|^2 dx =0.
\]
Therefore, $u$ is a constant function in $\bR^3$. From this and the fact that $u \in [L_{q,r}(\bR^3)]^3$, we conclude that $u \equiv 0$.  The proof is then completed.
\end{proof}
\section{Liouville theorems for solutions supported in strips in $\bR^3$} \label{strip-ses}

This section proves Theorem \ref{strip-thrm}. We follow the approach used in the proof of Theorem \ref{mixed-norm}.  Though the proof is similar to that of Theorem \ref{mixed-norm}, it is more involved and many important adjustments are needed. We therefore provide most of the essential steps in the proof. 
\begin{proof}[Proof of Theorem \ref{strip-thrm}]  For each $R>0$, let us denote $B_R'$ the ball in $\bR^2$ centered at the origin with radius $R$. Also, let $\phi \in C_0^\infty(\bR^2)$ be a standard cut-off function with $0 \leq \phi \leq 1$ and 
$$ \phi =1 \quad \text{on} \quad B'_{1/2} \quad \text{and} \quad \phi =0 \quad \text{on} \quad \bR^2 \setminus B'_1. $$
For each $R>2$, we defined $\phi_R(x') = \phi(\frac{x'}{R})$ for $x' \in \bR^2$. Then, it follows that
$$
\phi_R =1 \quad \text{on} \quad B'_{R/2}, \quad \text{and} \quad \phi_R =0 \quad \text{on} \quad \bR^2 \setminus B'_{R}.
$$
Moreover, there exists $N_0 >0$ such that
\[
|\nabla \phi_R| \leq \frac{N_0}{R}, \quad \text{and} \quad |D^2\phi_R| \leq \frac{N_0}{R^2} \quad \forall \ R >0.
\]
Let us also denote 
$$D_R = (B_R' \setminus B_{R/2}') \times [-R_0, R_0].$$
Arguing as in the proof of Theorem \ref{mixed-norm}, we can use $u (x) \phi_R(x')$ with $x = (x', x_3) \in \bR^3$  as a test function in \eqref{test.eqn} for the equation \eqref{NS.eqn} and obtain
\[
\sum_{k=1}^3 \int_{\bR^3}  \nabla u_k (x) \cdot  \nabla [\phi_R(x') u_k(x)] dx + \int_{\bR^3} (u \cdot \nabla u) \cdot [\phi_R (x') u(x)] dx - \int_{\bR^3}  p \text{div} [\phi_R (x') u(x) ] dx =0.
\]
Then, we use the assumption that $\textup{support}(u) \subset \bR^2\times [-R_0, R_0]$ to perform the integration by parts as in the proof of \eqref{2-identity} to obtain
\begin{equation*}
\int_{B'_{R} \times [-R_0, R_0]} |\nabla \otimes u|^2 \phi_R(x') dx = \frac{1}{2} \int_{D_R} |u|^2 \Delta \phi_R(x') dx + \int_{D_R} \Big[ (u_1, u_2) \cdot \nabla \phi_R (x')\Big] \Big[\frac{ |u|^2}{2} +p \Big] dx.
\end{equation*}
Therefore, it follows that 
\begin{align}  \nonumber
\int_{B'_{R/2} \times [-R_0, R_0]} |\nabla \otimes u|^2 dx & \leq \frac{1}{2}\int_{D_R} |u|^2 |\Delta \phi_R(x')| dx \\ \label{compact-id}
& \quad \quad \quad  + \frac{1}{2} \int_{D_R} |u|^3  |\nabla \phi_R(x')| dx + \int_{D_R} |p| |u| |\nabla \phi_R (x')| dx.
\end{align}
We now denote $J_1(R), J_2(R), J_3(R)$ the first, the middle and the last term in the right hand side of the estimate \eqref{compact-id}, respectively. To control the term $J_1(R)$ in \eqref{compact-id}, we observe that as $R>2$ , we have
\[
\omega_1(x')  \sim R^{-\alpha}, \quad \forall \ x'  = (x_1, x_2)\in B_{R}' \setminus B_{R/2}'.
\]
Therefore,
\[
J_1(R) \leq N R^{\frac{2\alpha}{q}}\int_{D_R} \Big[ |u(x)|^2  \omega_1(x')^{\frac{2}{q}} \Big] |\Delta \phi_R(x')| dx.
\]
 Then, for $q_1, r_1$ defined in \eqref{q-1-r-1.def}, we can use H\"{o}lder's inequality and perform the calculation as in the proof of \eqref{I-1-term} to obtain
\[
\begin{split}
J_1(R) & \leq N R^{\frac{2\alpha}{q}}\norm{\Delta \phi_R}_{L_{q_1, r_1} (D_R)} \norm{u}_{L_{q,r}(D_R, \omega)}^2 = N R^{\frac{2\alpha}{q} + \frac{2}{q_1}  -2} R_0^{ \frac{1}{r_1}} \norm{u}_{L_{q,r}(D_R, \omega)}^2 \\
& = N R^{\frac{2(\alpha -2)}{q}} R_0^{1-\frac{2}{r}}\norm{u}_{L_{q,r} (D_R, \omega)}^2  = N R^{-\frac{2}{3}} R_0^{1-\frac{2}{r}}\norm{u}_{L_{q,r} (D_R, \omega)}^2.
\end{split}
\]
This last estimate particularly implies that 
\[
\lim_{R\rightarrow \infty } J_1(R) =0.
\]
Now, in a similar way, we can also have
\[
J_2(R) \leq N R^{\frac{3\alpha}{q}}\int_{D_R}\Big[ |u(x)|^3 \omega(x')^{\frac{3}{q}} \Big] |\nabla \phi_R(x')| dx.
\]
 Then, with $q_2, r_2$  as in \eqref{q-2-r-2}, we can use the H\"{o}lder's inequality to infer that
\[
\begin{split}
J_2(R) & \leq NR^{\frac{3\alpha}{q}} \norm{\nabla \phi_R}_{L_{q_2, r_2}(D_R)}  \norm{u}_{L_{q,r}(D_R, \omega)}^3 = N R^{\frac{3\alpha}{q} + \frac{2}{q_2} -1} R_0^{\frac{1}{r_2}} \norm{u}_{L_{q, r}(D_R, \omega)}^3 \\
& = N R^{1+ \frac{3(\alpha-2)}{q}} R_0^{1-\frac{3}{r}} \norm{u}_{L_{q,r}(D_R)}^3  =  N R_0^{1-\frac{3}{r}} \norm{u}_{L_{q,r}(D_R ,\omega)}^3.
\end{split}
\]
As $u \in L_{q,r}(\bR^3)$, we obtain that
\[
\lim_{R\rightarrow \infty} J_2(R) =0.
\]
Finally, we control $J_3(R)$. This can be done exactly the same as the estimates of the other terms.  Again, we have
\[
J_3(R) \leq N R^{\frac{3\alpha}{q}}\int_{D_R} \Big[ u(x) \omega_1(x')^{\frac{1}{q}} \Big] \Big[ p(x) \omega_1(x')^{\frac{2}{q}} \Big] |\nabla \phi_R(x')| dx.
\]
Now, observe that since $q \in [3, 6]$,  $\alpha = \frac{6-q}{q} \in [0, 2)$. Hence, it follows from Lemma \ref{remark-weight} that $\omega_1 \in A_\frac{q}{2}(\bR^2)$. Therefore, we are able to apply Lemma \ref{pressure-lemma}. Then, we can perform the calculation using H\"{o}lder's inequality and Lemma \ref{pressure-lemma} to obtain
\[
\begin{split}
J_3(R) & \leq NR^{\frac{3\alpha}{q}} \norm{\nabla \phi_R}_{L_{q_2, r_2}(D_R)}  \norm{p}_{L_{\frac{q}{2}, \frac{r}{2}}(\bR^3, \omega)} \norm{u}_{L_{q, r}(D_R, \omega)}\\
& \leq N R^{\frac{3\alpha}{q} + \frac{2}{q_2} -1}  R_0^{\frac{1}{r_2}} \norm{u}_{L_{q, r}(\bR^3,\omega)} \norm{u}_{L_{q,r}(D_R, \omega)} \\
& = N  R_0^{1- \frac{3}{r}} \norm{u}_{L_{q, r}(\bR^3,\omega)} \norm{u}_{L_{q,r}(D_R, \omega)}.
\end{split}
\]
From this, and with the assumption that $u \in L_{q, r}(\bR^3)$, we obtain that
\[
\lim_{R\rightarrow \infty} J_3(R) =0.
\]
Collecting all of the estimates of $J_k(R)$ with $k =1, 2,3$ and using \eqref{compact-id}, we infer that
\[
\int_{\bR^3} |\nabla \otimes u|^2 dx = \lim_{R\rightarrow \infty} \int_{B_{R/2}' \times[-R_0, R_0]} |\nabla \otimes u|^2 dx =0.
\]
This estimate implies that $u$ is a constant function.  This together with the fact that $u \in [L_{q, r}(\bR^3 ,\omega)]^3$ and $\alpha \in [0, 2)$, we conclude that $u \equiv 0$.  The proof is completed.
\end{proof}
We conclude this session with the following remark.
\begin{remark} \label{weighted-thrm-rm} By combining the proofs of Theorem \ref{mixed-norm} and Theorem \ref{strip-thrm} and using  \eqref{Riesz-1}, we can see that the assertion of Theorem \ref{mixed-norm} also holds if $u \in L_q(\bR^3, \omega)$ where
\[
\norm{u}_{L_q(\bR^3, \omega)} = \left( \int_{\bR^3} |u(x)|^q \omega(x) dx \right)^{\frac{1}{q}} 
\]
for $q \in [3, \frac{9}{2}]$ and 
\[
\omega(x) = (1 + |x|)^{-(\frac{9}{q}-2)}, \quad \forall \ x\in \bR^3.
\]
\end{remark}
\section{Liouville theorems for solutions supported in cylinders} \label{cyl-section}

This session provides the proof of Theorem \ref{cyl-thrm}. Recall that for each $R >0$, we denote the cylinder along the $x_3$-axis in $\bR^3$ by
\[
C_R = B_R' \times \mathbb{R}, 
\]
where $B_R'$ is the ball in $\bR^2$ centered at the origin with radius $R$. The approach of the proof is similar to that of Theorem \ref{strip-thrm}. However, important details in the calculation needed to be adjusted.
\begin{proof}[Proof of Theorem \ref{cyl-thrm}] We use the approach as in the proof of Theorem \ref{strip-thrm}.  For each $R>0$.  Let $\phi \in C_0^\infty(\bR)$ be a standard cut-off function with $0 \leq \phi \leq 1$ and 
$$ \phi =1 \quad \text{on} \quad [-1/2, 1/2] \quad \text{and} \quad \phi =0 \quad \text{on} \quad \bR \setminus (-1, 1). $$
For each $R>2$, we defined $\phi_R(s) = \phi(\frac{s}{R})$ with $s \in \bR$. Then, it follows that
$$
\phi_R =1 \quad \text{on} \quad [-R/2, R/2] , \quad \text{and} \quad \phi_R =0 \quad \text{on} \quad \bR  \setminus (-R, R).
$$
Moreover, there exists $N_0 >0$ such that
\[
|\phi_R'| \leq \frac{N_0}{R}, \quad \text{and} \quad |\phi_R'| \leq \frac{N_0}{R^2} \quad \forall \ R >0.
\]
We also define 
$$E_R =  B_{R_0}' \times [ (-R, R) \setminus (-R/2, \times R/2)].$$
Arguing as in the proof of Theorem \ref{mixed-norm}, we can use $u (x) \phi_R(x_3)$ with $x = (x_1, x_2, x_3) \in \bR^3$ as the test function in \eqref{test.eqn} for the equations \eqref{NS.eqn}.  Then, we obtain
\begin{equation*} 
\sum_{k=1}^3 \int_{\bR^3}  \nabla u_k  \cdot  \nabla [\phi_R(x_3) u_k] dx + \int_{\bR^3} (u \cdot \nabla u) \cdot [\phi_R(x_3) u] dx - \int_{\bR^3} p \text{div} [\phi_R(x_3) u] dx =0.
\end{equation*}
From this, we can perform the integration by parts as in the proof of Theorem \ref{mixed-norm} using the fact that $\textup{support(u)} \subset C_{R_0}$ to obtain

\begin{equation*}
\int_{C_R} |\nabla \otimes u|^2 \phi_R(x_3) dx = \int_{E_{R}} |u|^2 \phi_R''(x_3)  dx  + \int_{E_R}  u_1(x) \phi_R'(x_3)  \Big[\frac{ |u|^2}{2} +p \Big] dx.
\end{equation*}
Therefore, it follows that
\begin{align}  \label{C-compact-id}
\int_{C_{R/2}} |\nabla \otimes u|^2 dx & \leq \frac{1}{2}\int_{E_R} |u|^2 |\phi_R''(x_3)| dx   + \frac{1}{2} \int_{E_R} |u|^3  |\phi_R'(x_3)| dx + \int_{E_R} |p| |u| |\phi_R'(x_3)| dx.
\end{align}
We then denote $K_1(R),  K_2(R),  K_3(R)$ the first, the second, and the last term in the right hand side of \eqref{C-compact-id}, respectively.  To control the term $K_1(R)$ in \eqref{C-compact-id}, we observe that as $R>2$, $\omega_2(x_3) \sim R^{-\alpha}$ for all $x_3$ such that $R/2< |x_3|<R$. Therefore, we have 
\[
K_1(R) \leq R^{\frac{2\alpha}{r}} \int_{E_R} |u(x)|^2 \omega_2(x_3)^{\frac{2}{r}} |\phi_R''(x_3)| dx.
\]
Then, for $q_1, r_1$ defined in \eqref{q-1-r-1.def}, we can use H\"{o}lder's inequality as in the proof of \eqref{I-1-term} to obtain
\[
\begin{split}
K_1(R) & \leq N R^{\frac{2 \alpha}{r}} \norm{\phi_R''}_{L_{q_1, r_1}(E_R)}  \norm{u}_{L_{q, r}(E_R, \omega)}^2 \\
& = N R_0^{\frac{2}{q_1}} R ^{-1+ \frac{1}{r_2} + \frac{2 \alpha}{r}}\norm{u}_{L_{q, r}(E_R, \omega)}^2 \\
& = N R_0^{2- \frac{4}{q}} R^{\frac{2(\alpha-1)}{r}} \norm{u}_{L_{q, r}(E_R, \omega)}.
\end{split}
\]
 As a consequence,  we see that
\[
\lim_{R\rightarrow \infty } K_1(R) =0.
\]
 Now, let $q_2, r_2$ be as in \eqref{q-2-r-2}. In a similar way as we we just did, we also have
\[
\begin{split}
K_2(R) & \leq N R^{\frac{3 \alpha}{r} }\int_{E_R} |u(x)|^3 \omega_2(x_3)^{\frac{3}{r}} |\phi_R'(x_3)| dx  \\
&  \leq N R^{\frac{3 \alpha}{r} } \norm{\phi_R'}_{L_{q_2, r_2}(E_R )}  \norm{u}_{L_{q, r}(E_R,\omega)}^3 \\
& = N R^{\frac{3 \alpha}{r} + \frac{1}{r_2} -1} R_0^{\frac{2}{q_2}} \norm{u}_{L_{q,r}(E_R, \omega)}^3  = N(q, R_0) R^{\frac{3 (\alpha-1)}{r}}  \norm{u}_{L_{q, r}(E_R, \omega)}^3.
\end{split}
\]
Therefore, we obtain that
\[
\lim_{R\rightarrow \infty} K_2(R) =0.
\]
Finally, we control $K_3(R)$ in the same fashion. We first observe that
\[
K_3(R) \leq N R^{\frac{3 \alpha}{r}}\int_{E_R} \Big[ |u(x)| \omega_2(x_3)^{\frac{1}{r}}\Big] \Big[ |p(x)| \omega_2(x_3)^{\frac{2}{r}} \Big] |\phi'_R(x_3)| dx.
\]
Now, we note that as $\alpha \in [0, 1)$, it follows from Lemma \ref{remark-weight} that $\omega_2 \in A_{\frac{r}{2}}(\bR)$. Therefore, we can use Lemma \ref{pressure-lemma} with this weight.
Then, as $q_2, r_2$ defined in \eqref{q-2-r-2}, we can use H\"{o}lder's inequality and Lemma \ref{pressure-lemma} to see that
\[
\begin{split}
K_3(R) & \leq NR^{\frac{3 \alpha}{r}} \norm{\phi_R'}_{L_{q_2, r_2}(E_R)}  \norm{p}_{L_{\frac{q}{2}, \frac{r}{2}}(\bR^3, \omega)} \norm{u}_{L_{q, r}(E_R, \omega)}\\
& \leq N R^{\frac{3 \alpha}{r} + \frac{1}{r_2} -1}  R_0^{\frac{2}{q_2}} \norm{u}_{L_{q,r}(\bR^3, \omega)} \norm{u}_{L_{q,r}(E_R, \omega)} \\
& = N(R_0, q)  R^{\frac{3(\alpha-1)}{r}}  \norm{u}_{L_{q, r}(\bR^3, \omega)} \norm{u}_{L_{q,r}(E_R, \omega)}.
\end{split}
\]
Then, we also obtain 
\[
\lim_{R\rightarrow \infty} K_3(R) =0.
\]
Collecting all of the estimates of $J_k(R)$ with $k =1, 2,3$ and using \eqref{compact-id}, we infer that
\[
\int_{\bR^3} |\nabla \otimes u|^2 dx = \lim_{R\rightarrow \infty} \int_{C_{R/2}} |\nabla \otimes u|^2 dx =0.
\]
This implies that $u$ is a constant function in $\bR^3$. From this and the fact that $u \in [L_{q,r}(\bR^3, \omega)]^3$ and with $\alpha \in [0,1)$, we conclude that $u \equiv 0$.  The proof is completed.
\end{proof}
\begin{remark} We note that we can not take $\alpha =1$ in the above proof because with  $\alpha =1$,  $\omega_2$ is not in $A_{\frac{r}{2}}(\bR)$.
\end{remark}


\begin{thebibliography}{m}
%
\bibitem{Chae-2018} D. Chae,  J.  Wolf, {\it On Liouville type theorem for the stationary Navier-Stokes equations},  arXiv:1811.09051.

\bibitem{Chae-Wolf} D. Chae,  J.  Wolf,  {\it On Liouville type theorems for the steady Navier-Stokes equations in $\bR^3$}. J. Differential Equations 261 (2016), no. 10, 5541-5560.

\bibitem{CJL-R} D. Chamorro, O. Jarrin, P.-G. Lemari\'{e}-Rieusset, {\it Some Liouville theorems for stationary Navier-Stokes equations in Lebesgue and Morrey spaces},  arXiv:1806.03003.

\bibitem{David}  D. V. Cruz-Uribe, J. M. Martell, and C. P\'{e}rez.  {\it Weights, extrapolation and the theory of Rubio de Francia,} volume 215 of Operator Theory: Advances and Applications. Birkh\"{a}user/Springer Basel AG, Basel, 2011.

\bibitem{RF}  J. Garc\'{i}a-Cuerva,  J. L. Rubio de Francia,  {\it Weighted norm inequalities and related topics}. North-Holland Mathematics Studies, 116. Notas de Matem\'{a}tica, 104. North-Holland Publishing Co., Amsterdam, 1985.

\bibitem{Dong-Kim}  H. Dong,  D. Kim, {\it On $L_p$-estimates for elliptic and parabolic equations with $A_p$ weights}. Trans. Amer. Math. Soc. 370 (2018), no. 7, 5081-5130

\bibitem{D-P} H. Dong and T. Phan, {\it Mixed norm $L_p$-estimates for non-stationary Stokes systems with singular VMO coefficients and applications}, 20 pp., submitted, arXiv:1805.04143.

\bibitem{D-K-18}  H. Dong,  D. Kim, {\it Weighted $L_q$-estimates for stationary Stokes system with partially BMO coefficients}.  J. Differential Equations 264 (2018), no. 7, 4603-4649.


\bibitem{Galdi} G.P. Galdi,  {\it An Introduction to the Mathematical Theory of the Navier-Stokes Equations: Steady-State Problems}, 2nd ed., Springer, 2011.

\bibitem{KTW} H. Kozono,  Y. Terasawa, Y. Wakasugi, {\it A remark on Liouville-type theorems for the stationary Navier-Stokes equations in three space dimensions}. J. Funct. Anal. 272 (2017), no. 2, 804-818.

\bibitem{Krylov}  N. V. Krylov,  {\it Parabolic equations with VMO coefficients in Sobolev spaces with mixed norms}. J. Funct. Anal. 250 (2007), no. 2, 521-558.


\bibitem{Seregin-book} G. Seregin, {\it Lecture notes on regularity theory for the Navier-Stokes equations}. World Scientific Publishing Co. Pte. Ltd., Hackensack, NJ, 2015.

\bibitem{Seregin-2018} G. Seregin,  {\it Remarks on Liouville type theorems for steady-state Navier-Stokes equations}. Algebra and Analysis 30 (2018), no. 2, 238-248.

\bibitem{Seregin-1916}   G. Seregin, {\it Liouville type theorem for stationary Navier-Stokes equations}. Nonlinearity 29 (2016), no. 8, 2191-2195.

\bibitem{Seregin-Wang} G. Seregin and W. Wang, {\it Sufficient conditions on Liouville type theorems for the 3D steady Navier-Stokes equations}, arXiv:1805.02227 [


\bibitem{Tsai} T.-P. Tsai,  {\it Lectures on Navier-Stokes equations}. Graduate Studies in Mathematics, 192. American Mathematical Society, Providence, RI, 2018.

\end{thebibliography}
\end{document}